\documentclass{elsarticle}

\usepackage[utf8]{inputenc}
\usepackage[T1]{fontenc}
\usepackage{graphicx}
\usepackage{amsmath,amsfonts,amsthm,amssymb,mathtools}
\usepackage{hyperref}
\usepackage{cleveref}
\usepackage{bbm}
\usepackage{subcaption}
\usepackage{dsfont} 
\usepackage{enumerate}
\usepackage{xcolor}

\newtheorem{thm}{Theorem}
\newtheorem{lem}[thm]{Lemma}
\newtheorem{defn}[thm]{Definition}
\newtheorem{prop}[thm]{Proposition}

\newtheorem{cor}[thm]{Corollary}
\newtheorem{rem}[thm]{Remark}
\newtheorem{con}[thm]{Conjecture}

\newcommand{\N}{\mathbb{N}}
\newcommand{\R}{\mathbb{R}}

\newcommand{\D}{\mathcal{M}_+}
\newcommand{\DCa}{\mathcal{D}_{\mathrm{Ca}}}
\newcommand{\Dco}{\mathcal{D}_{\mathrm{co}}}
\newcommand{\DCo}{\Dco}
\newcommand{\F}{\mathcal{F}}
\newcommand{\Fco}{\mathcal{F}^{co}}
\newcommand{\FMS}{\mathcal{F}^{MSz}}
\newcommand{\HR}{\prescript{*}{}{\mathbb{\R}}} 
\newcommand{\leqHR}{\mathbin{\prescript{*}{}{\leq}}} 

\newcommand{\setsep}{\;\mid\;}

\newcommand{\biggpars}[1]{\biggl(#1\biggr)}
\newcommand{\set}[1]{\{#1\}}

\newcommand{\biggset}[1]{\biggl\{#1\biggr\}}
\DeclareMathOperator{\supp}{supp} 
\newcommand{\undersetbrace}[2]{\underset{#2}{\underbrace{#1}}} 

\begin{document}
\begin{frontmatter}
\title{Remarks on the tail order on moment sequences}
\author
{Vincent Bürgin}
\ead{vincent.buergin@tum.de}

\author
{Jeremias Epperlein}
\ead{jeremias.epperlein@uni-passau.de}

\author
{Fabian Wirth}
\ead{fabian.lastname@uni-passau.de}

\address
{University of Passau, %
    Faculty of Computer Science and Mathematics, %
     Passau, Germany}
\date{\today}

\begin{abstract}
    We consider positively supported Borel measures for which all moments exist. On the set of compactly supported measures in this class a partial order is defined via eventual dominance of the moment sequences. 
    Special classes are identified on which the order is total, but it is shown that already for the set of distributions with compactly supported smooth densities the order is not total. In particular we construct a pair of measures with smooth density for which infinitely many moments agree and another one for which the moments alternate infinitely often.
    This disproves some recently published claims to the contrary. Some consequences for games with distributional payoffs are discussed.
\end{abstract}


\end{frontmatter}

\section{Introduction}

In this paper we analyze a particular order on real-valued random variables and related game-theoretic notions for games in which the payoffs are random variables. The impetus for analyzing this problem results from a sequence of recent papers which discuss a stochastic  order defined via the moment sequences of such random variables. In the course of the investigation several new properties of moment sequences are derived which we believe to be of interest in their own right in the context of the moment problem. In particular, moment sequences corresponding to distinct, compactly supported, $C^\infty$ densities can coincide infinitely often, or alternate. For piecewise analytic, compactly supported densities this is not possible.

In \cite{rassDecisionsUncertainConsequences2016a,rassDefendingAdvancedPersistent2017} an order on probability distributions which are compactly supported was introduced and applied in a game theoretic context. Theoretical support for these papers was provided in the arxiv papers \cite{rass2020gametheoretic_p1_v5,rass2020gametheoretic_p2_v2,rass2017gametheoretic_p3_v1}. The aim of these papers was to lay the foundations for a theory of games with probability distributions as payoffs. This theory was further developed in \cite{rassGameTheorySecurity2018} and has found applications in the literature related to energy distribution networks, \cite{alshawishQuasipurificationMixedGame2019,alshawishRiskMitigationElectric2019}.

The development may be briefly summarized in the following way. Consider probability distributions supported on an interval $[a,b]$, $0<a<b$ which have continuous densities, or which are discrete. For such distributions the moment sequences $s=(s_n)_{n\in\N}$ exist. We may define an order relation, called the tail order, by considering eventual domination of the moment sequences, i.e. by the requirement that after a certain index the moments of one sequence always exceed those of the other. A major tool in the theoretical analysis consists of mapping moment sequences to the space of hyperreals, i.e., the space of real sequences modulo a free ultrafilter. It is then attempted to derive properties of moment sequences from properties of the hyperreals.

Major claims in \cite{rassDecisionsUncertainConsequences2016a,rass2020gametheoretic_p1_v5} are the following:
\begin{enumerate}[1.]
    \item The tail order is total when restricted to compactly supported probability distributions which are discrete or have a continuous density.
    \item The map from the moment sequences to the hyperreals is injective. 
    \item If, for two probability distributions $\mu_1,\mu_2$ from the considered class, $\mu_2$ dominates $\mu_1$ in the tail order, then the density of $\mu_2$ dominates that of $\mu_1$ at the right end of the support. (Paraphrasing \cite[Theorem~2]{rassDecisionsUncertainConsequences2016a} or \cite[Theorem~2.15]{rass2020gametheoretic_p1_v5}.)
    \item It is concluded in the aftermath of \cite[Theorem~3.3]{rass2020gametheoretic_p1_v5} that games with probability distributions as payoffs have Nash equilibria in mixed strategies. The computation of such equilibria is discussed in \cite{rass2020gametheoretic_p2_v2}.
\end{enumerate}

We will show that the claims 1-4 are false in this generality. Claim 1 is refuted by Corollaries~\ref{cor:infinitely-many-agree} and~\ref{cor:not-total-cinfty-unimodal}; the next claim does not hold because of Proposition~\ref{prop:no-embedding}; Claim 3 is contradicted by Theorem~\ref{thm:nodominanceofdensitiesrequired} and, finally, Claim 4 does not hold by Theorem~\ref{thm:noNash} even in a simple case, where the tail order is restricted to the case of finite distributions, where it is total.

Despite this collection of negative results we were also interested in the question, to which extent the approach is viable. We will identify classes of positive measures on which the tail order is total. We will also identify a construction of hyperreal numbers which leads to an embedding of moment sequences. For the theory of games, however, we have no positive result, as already the easiest case leads to the nonexistence of equilibria.

The paper is organized as follows.
In the ensuing Section~\ref{sec:preliminaries} we recall some well-known results on moment sequences, moment generating functions and stochastic orders. The tail order is formally introduced. We also recall the Müntz-Szász theorem, which will prove instrumental in our constructions.
Section~\ref{sec:hyperreals} reviews some basic facts from nonstandard analysis concerning constructions of the hyperreals. We show that for the standard ultrafilters built from the Fréchet filter, we do not obtain an embedding of moment sequences into the hyperreals. For a more specific class of ultrafilters, extending what we call a Müntz-Szász filter, this property can be achieved. While this construction indeed allows for the definition of a total order on the moment sequences, we show that the order still depends on the ultrafilter and is therefore of little practical relevance.

In Section~\ref{sec:two_sufficient_conditions} we present two sufficient conditions of eventual dominance of moment sequences. In Proposition~\ref{prop:cdf} this condition is stated in terms of eventual dominance of the cumulative distribution function; Proposition~\ref{prop:density-signed} treats the signed case using eventual dominance of densities.  The first main result describes classes of measures of compact support for which the tail order is total. It will also be shown that such conditions are not necessary, that is, eventual dominance of cumulative distribution functions is not necessary for dominance in the tail order. 

The aim of Section~\ref{sec:continuousfunction-vanishingmoments} is to provide the main technical tool in our construction of counterexamples. It is shown that there exist signed, compactly supported measures with $C^\infty$ densities for which infinitely many moments vanish. Using this it can be shown that there are distinct compactly supported, positive $C^\infty$-densities for which infinitely many moments agree. By the Müntz-Szász theorem, the sequence where the moments agree cannot be a Müntz-Szász sequence, a term that will be made precise later. We conjecture that for any sequence which is not a Müntz-Szász sequence the result of this section can be replicated. But we were not able to show this.
Recently, the question of how many moments are necessary to reconstruct a measure has been investigated implicitly in \cite{melanova_recovery_2021} for finitely supported uniform measures.

Section~\ref{sec:pairs-of-measures} provides two more specific examples. These show that compactly supported measures with $C^\infty$ density or even $C^\infty$ unimodal densities need not be comparable in the tail order.

In Section~\ref{sec:games} we briefly comment on game theoretic aspects of games that have probability distributions as payoffs. One of the hopes associated with the original development of the tail order was that, based on this order, a theory of games can be derived that resembles the theory of finite games with real payoffs. In Section~\ref{sec:two_sufficient_conditions} some classes of  distribution-valued payoffs are identified, such that we have a totally ordered set. Essentially, in all these cases the order can be represented as a lexicographic order. While it is most likely folk knowledge in game theory that for games with lexicographically ordered payoffs, Nash equilibria need not exist, we were not able to locate an example to this effect in the literature. We show that even for the simplest case of games with payoffs that are finitely supported distributions, Nash equilibria need not exist.

\section{Preliminaries}
\label{sec:preliminaries}

Let $\N$ be the set of natural numbers including $0$. The real numbers are denoted by $\R$ and $\R_{\geq0}:=[0,\infty)$.
We denote the set of finite, positive Borel (or, equivalently, finite, positive Radon) measures supported on $\R_{\geq0}$ by $M_+(\R_{\geq0})$.
The moments of $\mu\in M_+(\R_{\geq0})$ are given by 
\begin{equation}
    \label{eq:moment}
    s_k (\mu) := \int_0^\infty t^k d\mu(t), \quad k\in \N,
\end{equation}
provided the integral exists. We will also use this notation for signed measures for which the respective integral exists.
The classic moment problem is the problem of identifying the sequences $s=(s_k)_{k\in\N}$ that can occur as moments of a particular measure. If this is the case then $s$ is called a \textit{moment sequence}. A moment sequence $s$ is called \emph{determinate} if there is a unique measure $\mu$ with $s=s(\mu)$. We refer to \cite{akhiezer1965classical,schmudgen2017moment} for background in this area.



 We are interested in positive measures for which all moments exist and define
\begin{equation}
    \label{set:all-moments-exist}
    \D:= \left\{ \mu \in M_+(\R_{\geq0})\;\middle\vert\; \text{ all moments $s_k(\mu), k\in\N,$ exist} \right\}.
\end{equation}
We note that all finite, compactly supported, positive measures are trivially elements of $\D$. An important subset of $\D$ can be stated using the \emph{Carleman condition}, see \cite[Theorem~4.3]{schmudgen2017moment}. We define the set as
\begin{equation}
    \label{set:Carleman}
    \DCa:= \left\{ \mu \in \D \;\middle\vert\;  \sum_{k=1}^\infty s_k(\mu)^{-\frac{1}{2k}}=\infty \right\}.
\end{equation}
The moment sequences of $\mu\in \DCa$ are determinate and for a compactly supported measure $\mu\in\D$ it is easy to see that $\mu\in\DCa$. Our main interest will be the set
\begin{equation}
    \label{set:compactlysupported}
    \DCo:= \left\{ \mu \in \D \;\middle\vert\;  \exists 0<a<b <\infty
    \;:\; \supp \mu \subset [a,b]\right\}.
\end{equation}
Associated to $\mu\in\D$ we define the formal series
\begin{equation}
    G_{\mu}(z) = \sum_{k=1}^\infty \frac{s_k(\mu)}{k!} z^k = \int_0^\infty e^{zx} d\mu(x).
\end{equation}
If this series converges in a neighborhood of $0$, then we speak of the \emph{moment generating function} of $\mu$. Closely related is the \emph{probability generating function}
\begin{equation}
    P_{\mu}(z) = \int_0^\infty z^{x} d\mu(x), \quad z\in[0,1],
\end{equation}
with the property that $P_{\mu}(e^{z})=G_{\mu}(z)$.

Stochastic orders are frequently used in applications, see \cite{shaked2007stochastic}. The classic order for probability measures supported on $[0,\infty)$ defines $\mu_1$ to be smaller than $\mu_2$, if for every increasing function $f:[0,\infty)\to\R$ we have
\begin{equation}
\label{eq:standardstochasticorder}
    \int_0^\infty f(x) \,d\mu_1(x) \leq \int_0^\infty f(x) \,d\mu_2(x).
\end{equation}
Other orders that have been suggested in the literature are the \emph{increasing concave order}, where the requirements of \eqref{eq:standardstochasticorder} need only hold for all increasing, bounded, concave functions, and the \emph{probability generating function order}, where we require that $P_{\mu_1}(t) \leq P_{\mu_2}(t)$ for all $t\in(0,1)$, see \cite{shaked2007stochastic,johnson2018stochastic}. A further order, considered in \cite{hutchcroft2020transience},
is the \emph{germ order}, which requires that there exists an $\varepsilon>0$ such that $P_{\mu_1}(t) \leq P_{\mu_2}(t)$ for all $t\in (1-\varepsilon,1)$.
As noted in \cite{hutchcroft2020transience} this order is total on $\DCa$ and corresponds to the lexicographic order on the alternating moment sequences $((-1)^{k+1}s_k)_{k\in\N}$. Finally, the moment order defined in Section 5.C.2 of \cite{shaked2007stochastic} is 
defined by the condition that one moment sequence dominates another one in every index and the \emph{moment generating function order} defined in Section 5.C.3 of \cite{shaked2007stochastic}
requires that one moment generating function dominates the other for all $z>0$. As we can see, several orders on probability measures have proven useful in the literature and it is not unheard of that for a particular class of probability measures an order is total.

The order we would like to discuss here originated with \cite{rassDecisionsUncertainConsequences2016a,rass2020gametheoretic_p1_v5} and we call it the \textit{tail order}. 

\begin{defn}
  On $\R^\N$ we define the \emph{tail relation} by
  \begin{equation}
     (x_n)_{n \in \N} \leq (y_n)_{n \in \N} \quad :\Leftrightarrow \quad 
     \exists n_0 \in \N:
  \forall n \geq n_0: x_n \leq
  y_n.
  \end{equation}
  The \emph{tail relation} on $\D$ is defined by
  \begin{equation}
     \mu_1 \preceq \mu_2 \quad :\Leftrightarrow \quad 
     s(\mu_1) \leq s(\mu_2).
  \end{equation}
\end{defn}

Note that the tail relation on $\R^\N$ is reflexive and transitive but not antisymmetric, thus a preorder or quasiorder. This translates to the tail relation on $\D$, as in $\D$ there are distinct measures with identical moment sequences, see \cite[Example~4.22]{schmudgen2017moment}.
We now aim to show that on $\Dco$ the tail relation is also antisymmetric, so that it defines a partial order.

To this end we recall the following seminal result from approximation theory, the Müntz-Szász theorem. The result 
has many versions, and we state one that best fits our
situation. See for example \cite[Chapter 4, Section 2, E.7 and E.9]{borwein1995}.
\begin{thm}[Müntz-Szász theorem in $C{[a,b]}$ and $L_p{[a,b]}$]\label{thm:muntz-szasz}
  Let $(n_k)_{k \in \N}$ be a strictly increasing sequence of nonnegative integers, let $0<a<b$ and $1\leq
  p\leq \infty$. Then
  the following are equivalent.
  \begin{enumerate}[(i)]
    \item $\text{span}\{x^{n_0},\dots,x^{n_k},\dots\}$ is dense in
      $C[a,b]$.
    \item $\text{span}\{x^{n_0},\dots,x^{n_k},\dots\}$ is dense in
      $L_p[a,b]$.
      \item $\sum_{k=0,n_k \neq 0}^\infty \frac{1}{n_k} = \infty$.
    \end{enumerate}
  \end{thm}

We will call strictly increasing sequences of nonnegative integers satisfying item (iii) of the previous theorem \emph{Müntz-Szász sequences}. An important and well-known consequence for the space $\Dco$ is the following.

\begin{cor}
\label{cor:msz-uniqueness}
Let $(n_k)_{k\in\N}$ be a Müntz-Szász sequence. For every nonnegative sequence $(s_{k})_{k\in\N}$ 
there is at most one  $\mu\in \Dco$
such that $s_{n_k}(\mu)=s_k,  k\in\N$.
\end{cor}

\begin{proof}
Let $\mu_1,\mu_2\in\Dco$ with $n_k$th moment equal to $s_k$, $k\in\N$.
Then the bounded linear functionals $l_1,l_2:C[a,b]\to\R$ induced by $\mu_1,\mu_2$ coincide on the space $X:=\text{span}\{x^{n_k}\;\vert\; k\in\N\}$. By Theorem~\ref{thm:muntz-szasz}, $X$ is dense in $C[a,b]$ and it follows that $l_1=l_2$. This implies $\mu_1=\mu_2$ and the proof is complete.
\end{proof}

With this we obtain:

\begin{cor}
\label{cor:antisymmetric}
    The tail relation $\preceq$ is antisymmetric on $\Dco$.
\end{cor}

\begin{proof}
Let $\mu_1,\mu_2\in \Dco$ with $\mu_1\preceq \mu_2\preceq \mu_1$. Then by definition there exists an index $n_0$, such that $s_n(\mu_1)=s_n(\mu_2)$ for all $n\geq n_0$. As $\{n\in\N\mid n\geq n_0\}$ defines a Müntz-Szász sequence, it follows from Corollary~\ref{cor:msz-uniqueness} that $\mu_1=\mu_2$ and the proof is complete.
\end{proof}



 


\section{Moments and hyperreal numbers}
\label{sec:hyperreals}

In this section we review basic concepts of ordering moment sequences by embedding into the hyperreal numbers. We show that without sufficient care two moment sequences can be identified in the hyperreals but for a specific type of ultrafilters a true embedding can be achieved. The idea to use hyperreals as a tool to generate orders on moment sequences originates with \cite{rass2020gametheoretic_p1_v5}. Our basic reference for the tools from nonstandard analysis that we use here is \cite{goldblatt2012lectures}.

As usual, $\R^{\N}$ denotes the space of real valued sequences indexed by the nonnegative integers. The Fréchet filter $\Fco$ on $\N$ consists of all cofinite sets. In other words, the Fréchet filter contains precisely the
complements of all sets of finite measure with respect to the counting measure on $\N$. The standard procedure is then to extend $\Fco$ to an ultrafilter $\mathcal{F}^*$ (i.e. a filter  with the property that for all $A\subset \N$ either $A$ or its complement lie in $\mathcal{F}^*$). This filter is automatically nonprincipal and any nonprincipal ultrafilter must contain $\Fco$.
The hypperreal numbers  are then defined to be $\HR:=\R^\N/\mathcal{F}^*$, where the equivalence class $[r]$ of a sequence $r\in \R^N$ is defined through the relation
\begin{equation}
\label{eq:hyperreal-equivalence}
    t\in [r] \quad \Leftrightarrow \left\{ n\in\N\mid t_n = r_n \right\} \in \mathcal{F}^*.
\end{equation}
Thinking of our moment sequences, this defines a map $\Phi_{\mathcal{F^*}}:\Dco\to\HR$ by
\begin{equation*}
    \mu \mapsto s(\mu) \mapsto [s(\mu)].
\end{equation*}
The problem is of course that in the last map the formation of equivalence classes depends on the ultrafilter $\mathcal{F}^*$, which is far from unique. 

It is thus of interest to note the following observation.

\begin{prop}
  \label{prop:no-embedding}
For a general nonprincipal ultrafilter $\mathcal{F}^*$ the map $\Phi_{\mathcal{F^*}}$ is not injective, even when restricted to $\left\{\mu\in\Dco \mid \text{$\mu$ has a $C^\infty$ density}\right\}$.  
\end{prop}

\begin{proof}
We will show in \Cref{cor:infinitely-many-agree} that there are $\mu_1,\mu_2\in\Dco$ with $C^\infty$ densities such that
\begin{equation}
    L := \left\{ n\in\N\mid s_n(\mu_1) = s_n(\mu_2) \right\}
\end{equation}
is an infinite set. The set $\Fco\cup \{L\}$ has the finite intersection property, i.e. finite intersections of its elements are nonempty. 
So the filter $\tilde{\F}$ generated by this set is proper. Then $\tilde{\F}$ may be extended to an ultrafilter $\F^*$. For the equivalence \eqref{eq:hyperreal-equivalence} we have $[s(\mu_1)]=[s(\mu_2)]$ because $L\in \F^*$.  
\end{proof}

We note however, that for more specific classes of ultrafilters an embedding may be obtained. At the heart of this construction lies the Müntz-Szász theorem (Theorem~\ref{thm:muntz-szasz}) and so for want of a better word we first introduce the \emph{Müntz-Szász filter} on $\N$. Consider the $\sigma$-additive measure $\vartheta$ on $\N$ defined by 
\begin{equation}
    \vartheta(\{0\}) = 0, ~~~ \text{and} ~ \vartheta(\{n\}) = \frac{1}{n}, ~ n \geq 1.
\end{equation}
Then the Müntz-Szász filter $\FMS$ on $\N$ is defined by
\begin{equation}
    L\in \FMS \quad :\Leftrightarrow \quad
    \vartheta(\N\setminus L) < \infty.
\end{equation}
It is clear that $\Fco\subset\FMS$ and so any ultrafilter containing $\FMS$ is nonprincipal and may be used to define a particular version of $\HR$. For these ultrafilters the map $\Phi_{\mathcal{F^*}}$ is well behaved.

\begin{thm}
\label{thm:injective-MS-Filter}
Let $\F^*$ be a nonprincipal ultrafilter on $\N $ containing $\FMS$. Then $\Phi_{\mathcal{F^*}}:\Dco\to\HR$ is injective.
\end{thm}

\begin{proof}
Let $\F^*$ be a nonprincipal ultrafilter on $\N $ containing $\FMS$ and let $\mu_1,\mu_2\in\Dco$ such that $[s(\mu_1)]=[s(\mu_2)]$, i.e.
\begin{equation*}
    L:=\left\{ n\in \N\mid s_n(\mu_1)=s_n(\mu_2) \right\}\in\F^*.
\end{equation*}
Since $\F^*$ is a proper filter, we have $(\N \setminus L) \not \in \F^* \supset \FMS$ and thus $\vartheta(L) =\infty$. By Corollary~\ref{cor:msz-uniqueness} it follows that $\mu_1=\mu_2$.
\end{proof}

As $\HR$ is an ordered field with order $\leqHR$, the previous result allows to pull back this order to obtain a total order on $\Dco$. To this end let $\F^*$ be an ultrafilter containing $\FMS$. Then \begin{equation}
\label{eq:ms-order-def}
    \mu_1 \preceq_{\F^*} \mu_2
    \quad :\Leftrightarrow \quad
    [s(\mu_1)] \leqHR\,  [s(\mu_2)]
\end{equation}
defines a total order on $\Dco$. It should be noted here that this order still depends on the ultrafilter $\F^*$ containing $\FMS$ as we will show in \Cref{rem:dependece-on-utrafiler}. So from this perspective we have not gained much.


\section{Sufficient conditions for eventual dominance of the moments}
\label{sec:two_sufficient_conditions} 

In this section we discuss two positive results. It is shown that eventual dominance of the cumulative distribution functions can be guaranteed by dominance conditions on the cumulative distribution function or on densities. We will also show that these conditions are not necessary.

\begin{prop}\label{prop:cdf}
  Let $F_1, F_2$ be the cumulative distribution functions of
  $\mu_1,\mu_2\in \Dco$. 
  If there exists $x_0>0$ such that  $F_1(x_0) > F_2(x_0)$ and such
  that
  $F_1(x) \geq F_2(x)$ for all $x \geq
  x_0$ then $\mu_1 \preceq \mu_2$ (notice that the order is reversed).
\end{prop}
\begin{proof}
  Let the supports of $\mu_1$ and $\mu_2$ be contained in
  $[a,b] \subset \R_{\geq 0}$.
  Since cumulative distribution functions are right-continuous and increasing, we
  can find $\varepsilon>0$ and $x_1 \in (x_0,b]$ such that
  $F_1(x) \geq \varepsilon \mathbbm{1}_{[x_0,x_1]}(x) +F_2(x)$ for all
  $x \in [x_0,b]$.
  Let $H$ be a continuously differentiable function on $[a,b]$ with
  positive
  derivative $h$. Then for $x
  \in [x_0,b]$ we have
  \begin{align*}
    \int_{t=a}^x 1 \,d\mu_1(t)
    &\geq \varepsilon \mathbbm{1}_{[x_0,x_1]}(x) +\int_{t=a}^x 1\,d\mu_2(t), \\
\intertext{and so}
    \int_{x=x_0}^b h(x)\int_{t=a}^x 1 \,d\mu_1(t)\,dx
    &\geq \varepsilon \int_{x=x_0}^bh(x)\mathbbm{1}_{[x_0,x_1]}(x)\,dx
      +\int_{x=x_0}^bh(x)\int_{t=a}^x 1 \,d\mu_2(t) \,dx,\\
    \int_{x=x_0}^b \int_{t=a}^x h(x) \,d\mu_1(t)\,dx
    &\geq \varepsilon \int_{x=x_0}^{x_1}h(x)\,dx
      +\int_{x=x_0}^b\int_{t=a}^x h(x) \,d\mu_2(t) \,dx.
  \end{align*}
  Reversing the order of integration gives
  \begin{multline*}
  \int_{t=a}^{x_0} \int_{x=x_0}^b  h(x) \,dx\,d\mu_1(t)+
    \int_{t=x_0}^{b} \int_{x=t}^b  h(x) \,dx\,d\mu_1(t)
    \\\geq \varepsilon (H(x_1)-H(x_0))+
    \int_{t=a}^{x_0} \int_{x=x_0}^b  h(x) \,dx\,d\mu_2(t)+
    \int_{t=x_0}^{b} \int_{x=t}^b  h(x) \,dx\,d\mu_2(t)
  \end{multline*}
  and thus
    \begin{multline*}
  \int_{t=a}^{x_0} H(b)-H(x_0) \,d\mu_1(t)+
    \int_{t=x_0}^{b} H(b)-H(t) \,d\mu_1(t)
    \\\geq \varepsilon (H(x_1)-H(x_0))+
    \int_{t=a}^{x_0} H(b)-H(x_0)\,d\mu_2(t)+
    \int_{t=x_0}^{b} H(b)-H(t)\,d\mu_2(t).
    \end{multline*}
    This can be simplified to
    \begin{multline*}
      H(b)F_1(b)-H(x_0)F_1(x_0)
      - \int_{t=x_0}^{b}H(t)\,d\mu_1(t)
      \\
      \geq \varepsilon (H(x_1)-H(x_0))+H(b)F_2(b)-H(x_0)F_2(x_0)
        - \int_{t=x_0}^{b}H(t)\,d\mu_2(t)
    \end{multline*}
    Using these estimates (where we set $H(x) = x^n$), we get the following estimate for the difference of the moments:
    \begin{align*}
      &\frac{1}{x_0^n}\left( \int_{t=a}^b t^n \,d\mu_2(t) -\int_{t=a}^b t^n
      \,d\mu_1(t)\right)\\
      &\geq-\mu_1([a,x_0])+\frac{1}{x_0^n}\left(\int_{t=x_0}^b t^n \,d\mu_2(t) -\int_{t=x_0}^b t^n
        \,d\mu_1(t)\right)\\
      &\geq-\mu_1([a,x_0])+\frac{\varepsilon(x_1^n-x_0^n)+b^n\left(F_2(b)-F_1(b))-x_0^n(F_2(x_0)-F_1(x_0)\right)}{x_0^n}\\
      \intertext{and using the fact that $F_1(b)=F_2(b)=1$ we obtain}        &\geq
          -\mu_1([a,x_0])+\varepsilon\left(\frac{x_1^n}{x_0^n}-1\right)
          +
          F_1(x_0)-F_2(x_0)
          \to \infty  \text{ for $n \to \infty$}.
    \end{align*}
    Therefore $\mu_1 \preceq \mu_2$.
  \end{proof}

The following lemma considers moments of compactly supported, signed measures with continuous density.

\begin{lem}\label{prop:density-signed}
  Let $f$ be a continuous functions with compact support contained in
  $\R_{\geq 0}$. If 
  there exists $x_0>0$ such that  $f(x_0) >0$ and such that $f(x) \geq 0$ for all $x \geq
  x_0$ then $\int_{\R} f(x) x^k\,dx> 0$ for all
  sufficiently large $k$.
  Furthermore, for every $t \in (0,x_0)$ we have
  $\frac{1}{t^k}\int_{\R} f(x) x^k\,dx \to \infty$ for $k \to \infty$.
\end{lem}
\begin{proof}
Let the support of $f$ be contained in
  $[a,b] \subset \R_{\geq 0}$.  Due to continuity we can find $x_1 \in (x_0,b]$ such that
  $f(x) > \varepsilon $ for all $x \in [x_0,x_1]$.
  We have
   \begin{align*}
     \frac{n+1}{x_1^{n+1}} \int_a^b f(x)x^n \,dx 
     &\geq\frac{n+1}{x_1^{n+1}} \left(\int_a^{x_0} f(x)x^n\,dx + \int_{x_0}^{x_1} f(x)x^n\,dx 
       \right) \\
     &\geq\frac{n+1}{x_1^{n+1}} \left(\int_a^{x_0} f(x)x^n\,dx + \int_{x_0}^{x_1} \varepsilon x^n\,dx\right) \\
     &\geq - \|f\|_\infty \frac{x_0^{n+1}-a^{n+1}}{x_1^{n+1}}  
     + \varepsilon
     \frac{x_1^{n+1} - x_0^{n+1}}{x_1^{n+1}}.
   \end{align*}
   As $x_0< x_1$, the first term in the last expression tends to $0$ and the second term tends to $\varepsilon$. This shows the eventual positivity of the moments of the measure with density $f$. The last claim now follows from the fact that
   $\frac{x_1^{n+1}}{t^{n+1}(n+1)}$ goes to infinity as $n$ goes to infinity.
\end{proof}
Using the preceding lemma we get a sufficient condition
for comparability of two measures with continuous density
in the tail order with respect to the densities.
For probability measures the proposition can also
be derived from \Cref{prop:cdf},
but we will need the slightly more general version later on.
A similar proposition can be found in \cite[Lemma 2.4]{rass2020gametheoretic_p1_v5}.
The proof provided there is problematic, so that we prefer to give a full proof here.
\begin{prop}\label{cor:density}
  Let $f_1,f_2$ be two non-negative continuous functions with compact
  support contained in
  $\R_{\geq 0}$. Assume $\int f_1 <\infty,  \int f_2 <\infty$. Let
  $\mu_1$ and $\mu_2$ be the measures with density $f_1$ and $f_2$
  with respect to the Lebesgue measure.
  If there exists $x_0>0$ such that  $f_1(x_0) < f_2(x_0)$ and such
  that $f_1(x) \leq f_2(x)$ for all $x \geq
  x_0$ then $\mu_1 \preceq \mu_2$.
\end{prop}
\begin{proof}
Apply \Cref{prop:density-signed} to $f_2-f_1$.
\end{proof}

\begin{thm}
\label{thm:sufficientconditionfororder}
    Let $\mathcal{F}$ be a class of piecewise continuous, compactly supported functions
    such that 
    for every pair $f_1,f_2 \in \mathcal{F}$ there are
    finitely many intervals covering the common support of $f_1$ and $f_2$
    such that on each of the given intervals
    one of the functions is larger or equal than the other.
    Then $\preceq$ is total when restricted to compactly supported absolutely continuous probability measures with densities in $\mathcal{F}$. 
\end{thm}
\begin{proof}
This is a direct consequence of \Cref{cor:density}.
\end{proof}

The following statement provides a list of examples of classes of measures in $\Dco$ for which the conditions of Theorem~\ref{thm:sufficientconditionfororder} are satisfied and for which the tail order is total. This extends \cite[Theorem~3]{rassDecisionsUncertainConsequences2016a}, where it is already shown that the tail order restricted to finitely supported measures corresponds to a lexicographic order of the probability mass vectors and is thus total on this class.

\begin{cor}
\label{cor:sufficientconditionfororder}
    The tail order $\preceq$ is total when restricted to
    probability measures with compact support in $\R_{\geq 0}$ and
    \begin{enumerate}[(i)]
        \item constant densities,
        \item polynomial densities,
        \item analytic densities,
        \item piecewise versions of the above.
    \end{enumerate}
\end{cor}

\begin{proof}
It is sufficient to consider the case of piecewise analytic densities as this case encompasses all the others. Here the claim follows from the identity theorem from complex analysis applied to the (finitely many) subintervals on which both functions are analytic.
\end{proof}

\begin{rem}
It is appropriate to point out that the use of Theorem~\ref{thm:sufficientconditionfororder} is not restricted to the cases described in Corollary~\ref{cor:sufficientconditionfororder}. Further examples can be constructed in the following way: Let $\phi:[c,d]\to[a,b]$ be a homeomorphism. Given a measure $\mu$ supported on $[a,b]$, we can pull it back via $\phi$ to a measure supported on $[c,d]$. Given a class of measures on $[a,b]$ which satisfies one of the conditions of Corollary~\ref{cor:sufficientconditionfororder}, then this yields a class of measures on $[c,d]$ which satisfies the assumptions of Theorem~\ref{thm:sufficientconditionfororder} but not necessarily those of the corollary.
\end{rem}

The two sufficient conditions in \Cref{prop:cdf} and \Cref{cor:density} do not provide characterizations of the totality of the tail order. Neither is necessary as soon as measures with infinite support are involved. The following result builds on \cite[Examples 4.22, 4.23]{buergin2020bachelorThesis}. 

\begin{figure}
    \centering
    \includegraphics[scale=0.8]{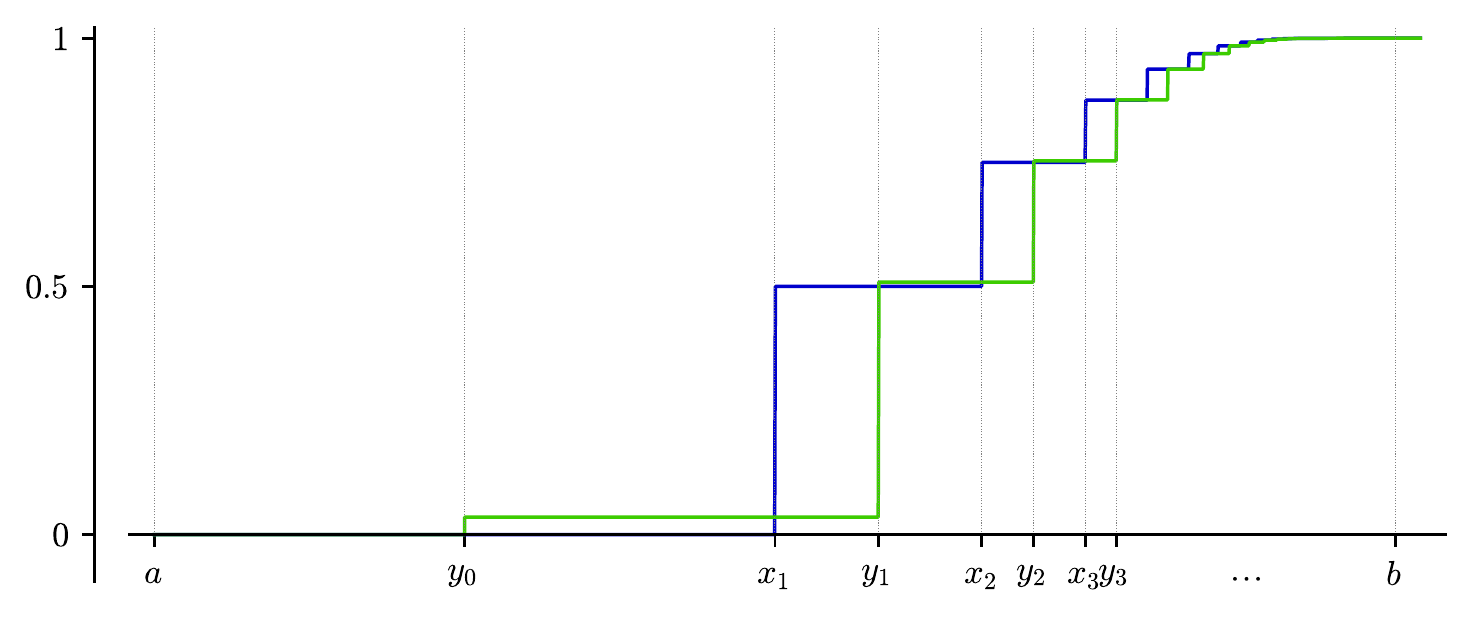}
    \caption{Cumulative distribution functions of the measures constructed in the proof of  
    \Cref{thm:nodominanceofdensitiesrequired}}
    \label{fig:my_label}
\end{figure}

\begin{thm}~
\label{thm:nodominanceofdensitiesrequired}
\begin{enumerate}[(i)]
    \item
    There exist discrete probability measures $\mu_f, \mu_g$ supported on a bounded, countably infinite subset of $\R_{\geq0}$ such that $\mu_f \preceq \mu_g$, and both the probability mass functions $f, g$ and the cumulative distribution functions $F,G$ alternate.
    \label{item:SufficientTailOrderConditionsNotNecessary-discretePart}
    
    \item
    There exist absolutely continuous probability measures $\mu_{\tilde{f}}, \mu_{\tilde{g}}$ supported on a compact subset of $\R_{\geq0}$ such that $\mu_{\tilde{f}} \preceq \mu_{\tilde{g}}$, both the probability density functions $\tilde{f}, \tilde{g}$ and the cumulative distribution functions $\tilde{F}, \tilde{G}$ alternate, and $\tilde{f}, \tilde{g}$ are continuous.
    \label{item:SufficientTailOrderConditionsNotNecessary-ACPart}
\end{enumerate}
\end{thm}
\begin{proof}
We first construct a counterexample for the discrete case. The same construction is then adapted for the absolutely continuous case.

(i) Let $1 < a < b$ and consider a set $X:=\set{x_i \mid i \in \N_{\geq 1}} \subset [a, b)$, indexed in strictly increasing order.
We start with a discrete distribution $\mu_f$ 
that gives positive probability mass precisely to the points in 
$X$ and with probability mass function $f$ which is strictly decreasing on $X$, i.e., $ f(x_i) > f(x_j)$ if $i<j$. We construct $\mu_g$ with mass function $g$, where each support point is shifted slightly to the right, with its probability adjusted to be only a little smaller, giving the remaining probability mass to the point $y_0 \coloneqq \frac{1+a}{2}$.
First, set $y_i \coloneqq \frac{x_i+ x_{i+1}}{2}, i \in \N_{\geq 1}$.
We adjust the probability assigned to $y_i$ by a factor $c_i < 1$, i.e. $g(y_i) = c_i f(x_i)$, chosen such that the following properties are satisfied:
\begin{enumerate}
    \item $c_i \geq \frac{x_i}{y_i}$,
    \item $c_i \geq 1-\frac{1}{2^i}$.
\end{enumerate}
As we will show shortly, the first property ensures that $\mu_g$ has greater moments than $\mu_f$, and the second property ensures that the distribution functions alternate.
In summary, we define $g$ as follows:

\begin{gather*}
    g(y_i) = c_i f(x_i), \quad
    c_i \coloneqq \max\biggset{\frac{x_i}{y_i}, 1-\frac{1}{2^i}
    } < 1, \quad
    g(y_0) = \sum_{i \geq 1} (1-c_i) f(x_i).
\end{gather*}
The construction is sketched in \Cref{fig:my_label}.
Obviously the mass functions $f, g$ alternate, because $f(x_i) > 0, g(x_i) = 0$ and $g(y_i) > 0, f(y_i) = 0$ for all $i \in \N_{\geq 1}$. We now show that the moments of $\mu_g$ dominate the moments of $\mu_f$, and the distribution functions alternate.
The first property of the $c_i$ implies that $\forall i \in \N_{\geq 1}: g(y_i)y_i  \geq  f(x_i)x_i$, so for the moments we have:
\begin{multline*}
    s_n(\mu_g) - s_n(\mu_f) \\
    = g(y_0)y_0^n + \sum_{i \geq 1} g(y_i) y_i^n - f(x_i) x_i^n
    \geq g(y_0)y_0^n + \sum_{i \geq 1} g(y_i) y_i x_i^{n-1} - f(x_i) x_i^n  \\
    = g(y_0)y_0^n + \sum_{i \geq 1} x_i^{n-1} \undersetbrace{(g(y_i)y_i - f(x_i)x_i)}{\geq 0}
    \geq g(y_0)y_0^n > 0.
\end{multline*}
Secondly, the distribution functions alternate: On the one hand for all $k \in \N_{\geq 1}$, 
\begin{multline*}
    G(y_k) 
    = g(y_0) + \sum_{i=1}^k c_i f(x_i) 
    = g(y_0) - \biggpars{\sum_{i=1}^k (1-c_i)f(x_i)} + F(x_k) \\
    = \biggpars{\sum_{i \geq k+1} (1-c_i)f(x_i)} + F(x_k) > F(x_k) = F(y_k).
\end{multline*}
On the other hand, 
\begin{multline*}
    G(x_k)
    = g(y_0) + \sum_{i=1}^{k-1} c_i f(x_i) 
    = g(y_0) + \sum_{i=1}^{k-1} f(x_i) - \sum_{i=1}^{k-1} (1-c_i) f(x_i) \\
    = F(x_k) - f(x_k) + g(y_0) - \sum_{i=1}^{k-1} (1-c_i) f(x_i)  
    = F(x_k) - f(x_k) + \sum_{i \geq k} (1-c_i) f(x_i).
\end{multline*}
Using the second property of the $c_i$, the strict monotonicity of $f$ on its support, and the geometric series identity $\sum_{i \geq k} \frac{1}{2^i} = 
\frac{1}{2^{k-1}}$, we further get:
\begin{gather*}
    \sum_{i \geq k} (1-c_i)f(x_i) \leq \sum_{i \geq k} \frac{1}{2^i} f(x_i) < \sum_{i \geq k} \frac{1}{2^i} f(x_k) = \frac{1}{2^{k-1}} f(x_k) \leq f(x_k).
\end{gather*}
This shows that $G(x_k) < F(x_k)$. 
So in summary we have $\mu_f \prec \mu_g$, and their probability mass, as well as distribution functions alternate.

(ii) The discrete construction can be easily adapted to the continuous case, see \cite[Example 4.23]{buergin2020bachelorThesis} for full details. For each $k$, let $z_k = \frac{y_{k-1} + x_k}{2}$. We shift the probability mass that $\mu_f$ puts on $x_k$ to the interval $[z_k, x_k]$, and the probability mass that $\mu_g$ puts on $y_k$ to the interval $[y_k, z_{k+1}]$, by defining densities
\begin{gather}
    \tilde{f} = \sum_{k \geq 1}  f(x_k) h_{[z_k, x_k]}, \quad
    \tilde{g} = \sum_{k \geq 0}  g(y_k) h_{[y_k, z_{k+1}]},
    \label{eq:acSufficientTailOrderConditionCounterexample-densitiesDefinition}
\end{gather}
where the $h_{[d, e]}$ are appropriately scaled versions of a continuous probability density $h$ which are supported on $[d, e]$, respectively.
Since the probability mass of $\mu_f$ is shifted to the left and the probability mass of $\mu_g$ is shifted to the right, it can be shown that the new absolutely continuous measures $\mu_{\tilde{f}}, \mu_{\tilde{g}}$ still satisfy $\mu_{\tilde{f}} \prec \mu_{\tilde{g}}$. The distribution functions still alternate, because
\begin{gather*}
    \tilde{F}(x_k) = F(x_k) > G(x_k) = \tilde{G}(x_k),\quad
    \tilde{G}(z_{k+1}) = G(y_k) > F(y_k) = \tilde{F}(z_{k+1}).
\end{gather*}
The densities $\tilde{f}, \tilde{g}$ are continuous if the series in \eqref{eq:acSufficientTailOrderConditionCounterexample-densitiesDefinition} converge uniformly. This can be achieved if the probability masses $f(x_k)$ approach zero asymptotically faster than the interval lengths. For example, we can use $x_k = 2-\frac{1}{k+1}$ and $f(x_k) = \frac{1}{2^k}$.

\end{proof}

\section{Signed measures with infinitely many vanishing moments}
\label{sec:continuousfunction-vanishingmoments}

In this section we provide the counterexamples that are used in the proof of Proposition~\ref{prop:no-embedding}.
Our goal is thus to construct a pair of probability measures
with support $[a,b]$, $0<a<b$, with continuous densities with respect
to the Lebesgue measure for
which infinitely many moments agree. To do so, we construct
a signed measure with support $[a,b]$ and continuous density for which
infinitely
many moments vanish.

 Let $(n_k)_{k \in \N}$ be a Müntz-Szász sequence.
  Using a standard Hahn-Banach argument found e.g. in \cite{giraudo2020mathoverflow}
  one can construct a signed measure
  with support on a compact
  interval of positive reals such that
  its $n_k$th moments vanish. Using the $L_1$ version of the Müntz-Szász theorem (\Cref{thm:muntz-szasz}) one can also ensure that the obtained signed measure has an $L^\infty$ density.

  To strengthen this to a continuous density, we apparently
  need a more hands-on approach, as the dual pair $(C[a,b],C[a,b])$ is
  not very accessible to standard functional analytic techniques.
  Also notice that, as mentioned in \Cref{cor:sufficientconditionfororder} we cannot
  hope for a polynomial or analytic density here.
  
The following lemma is the key to this construction.
\begin{lem}\label{lem:first-k-moments}
  For every $a,b \in \R$ with $0<a<b$ and $n \in \N$ there is a $C^\infty$
  function $f$ with support contained in $[a,b]$ and values in $[-1,1]$ such that
  $\int_a^b f(x) x^k \,dx = 0$ for all $k \in \{0,\dots,n\}$
  and such that there is $x_0 \in (a,b)$ with $f(x_0)>0$ and $f(x) \geq 0$
  for $x \in [x_0,b]$.
\end{lem}
\begin{proof}
  Let $g_1,\dots,g_{n+2}$ be non-trivial $C^\infty$ functions with compact
  pairwise disjoint support in $[a,b]$ and values in $[0,1]$.
  Our function $f$ will be of the form $\sum_{i=1}^{n+2} c_ig_i$ with
  $c_1,\dots,c_{n+2} \in \R$.
  The condition $\int_a^b f(x) x^k \,dx = 0$ for $k \in
  \{0,\dots,n\}$ gives rise to $n+1$ linear equations in the $n+2$
  variables $c_1,\dots,c_{n+2}$ and thus has a non-trivial solution.
  Since $g_1,\dots,g_{n+2}$ have disjoint support, the resulting
  function $f$ is also non-trivial. Since it is bounded, we can scale
  it to only obtain values in $[-1,1]$. By possibly replacing the
  function
  with its negation, we can ensure that the last non-zero coefficient
  among $c_1,\dots,c_{n+2}$ is positive.
\end{proof}

\begin{thm}\label{thm:continuous-infinitely-many-moments}
  For every $a,b \in \R$ with $0<a<b$ there exists a signed measure
  with integral $0$, support contained in $[a,b]$ and $C^\infty$ density for which infinitely
  many moments vanish. 
\end{thm}
\begin{proof}
  Let $(t_i)_{i \in \N}$ be a strictly increasing sequence in
  $[a,b]$ converging to $b$.
  Our density will be of the form $f=\sum_{i=0}^{\infty} g_i$
  where $g_i$ is $C^\infty$ with support in $(t_i,t_{i+1})$ and
  values in $[-c_i,c_i]$. We will inductively construct the functions
  $g_n$ and a strictly increasing sequence of exponents $k_n$ such that
  the moment sequences of the partial sums $h_n:=\sum_{i=0}^n g_i$ are eventually positive or eventually negative and, in addition, $
  0=\int_a^b h_n(x)x^{k_j} \,dx$ for all $j \in \N$ and $j\leq n$.
  
  We set $k_0:=1$ and $g_0:=0$.
  Now let $n>0$ and assume we have defined $k_0,\dots,k_n$ and
  $g_0,\dots,g_{n}$.
  By \Cref{lem:first-k-moments} we can find a $C^\infty$ function
  $\tilde{g}_{n+1}$ with support contained in $(t_{n+1},t_{n+2})$,
  whose first $k_n$ moments vanish, and whose first $n$ derivatives are
  bounded in absolute value by $\frac{1}{n}$. The latter condition will be used to
  ensure that all derivatives of $f$
  vanish in $b$.
   Moreover, by construction, Lemma~\ref{lem:first-k-moments}, and \Cref{prop:density-signed}, the moments of $\tilde{g}_{n+1}$
  are eventually positive.
  Now consider
  \begin{align*}
    a_\ell:=\frac
    {\frac{1}{t_{n+1}^\ell}\int h_n(x) x^\ell\,dx}
    {\frac{1}{t_{n+1}^\ell}\int \tilde{g}_{n+1}(x) x^\ell\,dx}.
  \end{align*}
  The numerator of this fraction converges to $0$ as $\ell \to \infty$ and (again by \Cref{prop:density-signed})
  the denominator goes to infinity.
  Hence there is $\ell_0 >k_n$ such that $|a_{\ell_0}|<1$.
  Now set $g_{n+1}:=-a_{\ell_0} \tilde{g}_{n+1}$
  and set $k_{n+1}:=\ell_0$.
  Then $\int_a^b \sum_{i=0}^{n+1} g_i(x)x^{k_j}\,dx=0$
  for $j\leq n$ by the induction hypothesis and our choice of
  $\tilde{g}_{n+1}$.
  We also have
  \begin{align*}
    \int_a^b \sum_{i=0}^{n+1} g_i(x)x^{k_{n+1}}\,dx &= \int_a^b h_n(x)
                                                  x^{\ell_0}\,dx -a_{\ell_0}
                                                  \int_a^b
                                                  \tilde{g}_{n+1}(x)x^{\ell_0}\,dx
    =0. 
  \end{align*}
Finally, by a further application of \Cref{prop:density-signed}, the moments of $h_{n+1}$ are all eventually positive or
  negative, because by construction the behaviour of $h_{n+1}$ at the right end of its support is determined by $g_{n+1}$. The claim for $f=\sum_{i=0}^\infty g_i$ now follows from an application of the theorem of dominated convergence.
  \end{proof}
  
The following result is the essential technical building block in the proof of \Cref{prop:no-embedding}. It does not yet answer the question of totality of the tail order.  

\begin{cor}
\label{cor:infinitely-many-agree}
  For every $a,b \in \R$ with $0<a<b$ there exist two distinct probability measures
  $\mu_{1}$ and $\mu_{2}$
  with support $[a,b]$ and $C^\infty$ densities $f_1$ and $f_2$ for which infinitely
  many moments agree. 
\end{cor}
\begin{proof}
  Let $g$ be a non-negative $C^\infty$ function with support contained
  in $[a,b]$ and integral $1$. Let $[c,d] \subset [a,b]$ be an interval of positive length
  such that $g(x)\geq \varepsilon$ for all $x \in [c,d]$ for some
  $\varepsilon>0$. Let $h$ be the density of a signed measure with the
  properties guaranteed by \Cref{thm:continuous-infinitely-many-moments},
  support contained in $[c,d]$ and values in $[-\varepsilon,\varepsilon]$.
  Then we can set $f_1:=g-h$ and $f_2:=g+h$. Both functions
  are non-negative, have integral $1$ and the corresponding measures have
  infinitely many moments in common.  
\end{proof}

In the statement of Corollary~\ref{cor:infinitely-many-agree}, we do not have any control over the particular exponents for which the statement holds. In light of the Müntz-Szász theorem (\Cref{thm:muntz-szasz}) we
conclude with the following conjecture.
  \begin{con}
    Let $a,b \in \R$ with $0<a<b$.
    Let $(n_k)_{k \in \mathbb{N}}$ be a strictly increasing sequence of positive integers such that
    $\sum_{k=0}^\infty \frac{1}{n_k} < \infty$.
    There exists a non-trivial signed measure with support contained in $[a,b]$
    and continuous density with respect to the Lebesgue measure
    for which all $n_k$th moments vanish.
  \end{con}
  
We can also frame the results in this section in terms of reconstruction from the knowledge of a sparse set of moments. More precisely, one can ask how many moments are needed in order to uniquely reconstruct a measure. For discrete uniform measures with support on $n$ points in the positive reals,  \cite[Proposition 24]{melanova_recovery_2021} showed that $n$ moments suffice. Our results show that for absolutely continuous probability measures with $C^\infty$ density even the knowledge of infinitely many moments can be insufficient for a reconstruction.

\section{Pairs of Measures with continuous densities and alternating moment sequences}
\label{sec:pairs-of-measures}

Our goal in this section is to construct two
continuous functions $f,g$ on $\R$ with support contained in $[a,b]$ such
that the measures $\mu_f$ and $\mu_g$ with density $f$ and $g$ with respect to the
Lebesgue measure are not comparable in the tail order. This achieves one of the main aims of this paper, namely to show that the tail order is not total even for benign classes of probability measures.

We will construct two different examples in order to show that even severely restricting the class of densities under consideration does not make the tail order total.
In the first example the densities are $C^\infty$, in the second example
the densities are also unimodal, in the sense that they are continuous, have compact support and one unique local maximum. In particular, the individual maps do not show any oscillatory behaviour, only their difference does.

\begin{thm}
\label{thm:cinfty-notcomparable}
For every $a,b \in \R$ with $0<a<b$ there exist two absolutely continuous probability measures $\mu_f, \mu_g$ supported on $[a,b]$ with 
$C^\infty$ densities $f,g$ and a strictly increasing sequence of positive integers
$(\ell_n)_{n \in \N}$ such that 
$s_{\ell_n}(\mu_f) < s_{\ell_n}(\mu_g)$ for odd $n \in \N$ and
$s_{\ell_n}(\mu_f) > s_{\ell_n}(\mu_g)$ for even $n \in \N$.
\end{thm}
\begin{proof}
Let $(t_n)_{n \in \N}$ be an increasing sequence in $[a,b]$
converging to $b$ with $t_0=a$. For $n \in \N$ let $h_n$ be a $C^\infty$ function with support 
$[t_n,t_{n+1}]$ and values in $[0,1]$ such that $h_n(x)>0$ for 
all $x \in (t_n,t_{n+1})$. Set $D\coloneqq \int_\R h_0(x) \,dx$.

Our functions $f$ and $g$ will have the form
\begin{align*}
  f&\coloneqq dh_0 + \sum_{i=1}^\infty c_{2i-1} h_{2i-1} + \frac{2}{3}
  \sum_{i=1}^\infty c_{2i} h_{2i}, \\
  g&\coloneqq d'h_0 + \frac{2}{3} \sum_{i=1}^\infty c_{2i-1}h_{2i-1} + 
  \sum_{i=1}^\infty c_{2i} h_{2i} 
\end{align*}
for coefficients $d,d', c_i \in [0,\infty)$, $i\in\N_{\geq 1}$.
We will inductively choose decreasing coefficients $c_n$ and an
increasing sequence of exponents $\ell_n$ in order to ensure that
$s_{\ell_n}(\mu_f) < s_{\ell_n}(\mu_g)$ for odd $n \in \N$ and
$s_{\ell_n}(\mu_f) > s_{\ell_n}(\mu_f)$ for even $n \in \N$. Finally we will choose
$d,d'$ such that the integral over $f$ resp. $g$ equals one.
In particular $d,d'$ will lie in $[0,\frac{1}{D}]$. To get started,
set $\ell_0\coloneqq 1$ and $c_1\coloneqq \min\{\frac{1}{2(b-a)},1\}$.

Now let $n=2k$ be even and assume we already defined $c_1, \dots,
c_{2k-1}$ and $\ell_0, \dots,\ell_{2k-2}$.
Choose $\ell_{2k-1} > \ell_{2k-2}$ such that
\begin{gather}
  \begin{aligned}
    &\int_a^{t_{2k}} \left(
      0 \cdot h_0(x)
      +
      \sum_{i=1}^k c_{2i-1}h_{2i-1}(x)
      +
      \frac{2}{3} \sum_{i=1}^{k-1} c_{2i} h_{2i}(x)
    \right) x^{\ell_{2k-1}}\,dx \\
    >&\int_a^{t_{2k}} \left(
        \frac{1}{D} \cdot h_0(x)
        +
        \frac{2}{3} \sum_{i=1}^k c_{2i-1}h_{2i-1}(x)
        +
        \sum_{i=1}^{k-1} c_{2i}h_{2i}(x)
      \right)x^{\ell_{2k-1}}\,dx.
  \end{aligned}
\end{gather}
This is possible by \Cref{cor:density}, since for all $x \in (t_{2k-1},t_{2k})$ we have, by the assumption on the supports of the $h_i$, that
\begin{gather}
  \begin{aligned}
    \MoveEqLeft
      \sum_{i=1}^k c_{2i-1}h_{2i-1}(x)
      +
      \frac{2}{3} \sum_{i=1}^{k-1} c_{2i} h_{2i}(x)
    =c_{2k-1}h_{2k-1}(x) \\
    &>\frac{2}{3}c_{2k-1}h_{2k-1}(x) \\
    &=\frac{1}{D} \cdot h_0(x)
        +
        \frac{2}{3} \sum_{i=1}^k c_{2i-1}h_{2i-1}(x)
        +
        \sum_{i=1}^{k-1} c_{2i}h_{2i}(x).
  \end{aligned}
\end{gather}
In the next step, choose $c_{2k} \in (0,c_{2k-1})$ small enough such that
\begin{align}
    &\int_a^{t_{2k}} \left(
      \sum_{i=1}^k c_{2i-1}h_{2i-1}(x)
      +
      \frac{2}{3} \sum_{i=1}^{k-1} c_{2i} h_{2i}(x)
    \right) x^{\ell_{2k-1}}\,dx \label{eq:odd}\\
    &>\int_a^{t_{2k}} \left(
        \frac{h_0(x)}{D} 
        +
        \frac{2}{3} \sum_{i=1}^k c_{2i-1}h_{2i-1}(x)
        +
        \sum_{i=1}^{k-1} c_{2i}h_{2i}(x)
      \right)x^{\ell_{2k-1}}\,dx + c_{2k} \int_{a}^b x^{\ell_{2k-1}} \,dx. \nonumber
  \end{align}
In the same way if $n=2k+1$ is odd, we can choose $\ell_{2k} >
\ell_{2k-1}$ and $c_{2k+1} \in (0,c_{2k})$ such that
\begin{align}
    \MoveEqLeft\int_a^{t_{2k+1}} \left(
       \frac{1}{D} \cdot h_0(x)
        +
      \sum_{i=1}^k c_{2i-1}h_{2i-1}(x)
      +
      \frac{2}{3} \sum_{i=1}^{k} c_{2i} h_{2i}(x)
    \right) x^{\ell_{2k}}\,dx+ c_{2k+1} \int_{a}^b x^{\ell_{2k}} \,dx.\label{eq:even} \\\nonumber
    <&\int_a^{t_{2k+1}} \left(       
        \frac{2}{3} \sum_{i=1}^k c_{2i-1}h_{2i-1}(x)
        +
        \sum_{i=1}^{k} c_{2i}h_{2i}(x)
      \right)x^{\ell_{2k}}\,dx. 
\end{align}

After defining the indices $\ell_n$ and coefficients $c_n$, define $d$
and $d'$ by
\begin{align*}
  d&\coloneqq\frac{1}{D}(1-\int_{a}^{b} \sum_{i=1}^\infty c_{2i-1}h_{2i-1}(x) + \frac{2}{3}
  \sum_{i=1}^\infty c_{2i}h_{2i}(x)\,dx), \\
  d'&\coloneqq\frac{1}{D}(1-\int_{a}^{b} \frac{2}{3}\sum_{i=1}^\infty
      c_{2i-1}h_{2i-1}(x) + 
  \sum_{i=1}^\infty c_{2i}h_{2i}(x)\,dx).
\end{align*}

By letting our coefficients $c_n$ go to zero fast enough, we can also
ensure that $f$ and $g$ are $C^\infty$, the only potentially problematic
place being $b$.

The claim that $s_{\ell_n}(\mu_f) < s_{\ell_n}(\mu_g)$ for odd $n \in \N$ and
$s_{\ell_n}(\mu_f) > s_{\ell_n}(\mu_f)$ for even $n \in \N$ now
follows directly from \eqref{eq:odd} and \eqref{eq:even}
together with the fact that the coefficients $c_n$ are decreasing.
\end{proof}

\begin{rem}\label{rem:dependece-on-utrafiler}
By a slight modification of the construction in the proof of \Cref{thm:cinfty-notcomparable} we can ensure that 
\begin{align*}
    s_{\ell_{2k+1}+i}(\mu_f) &< s_{\ell_{2k+1}+i}(\mu_g), \\
    s_{\ell_{2k}+i}(\mu_f) &> s_{\ell_{2k}+i}(\mu_g)
\end{align*}
for $k \in \N$ and $i \in \{0,\dots,\ell_{2k+1}\}$ or $i \in \{0,\dots,\ell_{2k}\}$, respectively.
Since $\sum_{i=0}^{n} \frac{1}{n+i} \geq \frac{1}{2}$, the increasing enumeration of both sets
\begin{align*}
    M_1&:=\{n \in \N \setsep s_n(\mu_f) < s_n(\mu_g)\},\\
    M_2&:=\{n \in \N \setsep s_n(\mu_g) > s_n(\mu_g)\}
\end{align*}
produces Müntz-Szász sequences. 
Let $K \in \FMS$ be a set in the Müntz-Szász filter. 
Then $M_i \not\subset \N \setminus K$ and hence $M_i \cap K \neq \emptyset$ for $i=1,2$.
Therefore $\FMS \cup \{M_1\}$ and $\FMS \cup  \{M_2\}$ can both be extended to 
nonprincipal ultrafilters $\F_1^*$ and $\F_2^*$, respectively.
Then $\mu_f$ is smaller than $\mu_g$ with respect to the order defined by
$\F_1^*$ but with respect to $\F_2^*$ the order is reversed. 
\end{rem} 

\begin{rem}
The result of Theorem~\ref{thm:cinfty-notcomparable} has an interesting but negative consequence when interpreted in the context of games, see \Cref{sec:games}. Consider in the notation of the theorem the densities
\begin{align*}
  \tilde{f}_1(x)&\coloneqq \gamma_1 h_0(x)
       +
       \frac{1}{3}\sum_{i=1}^kc_{2i-1}h_{2i-1}(x)
       +
       \frac{1}{3}\sum_{i=1}^\infty c_{2i}h_{2i}(x),\\
  \tilde{f}_2(x)&\coloneqq \gamma_2 h_0(x)
       +
       \frac{5}{3}\sum_{i=1}^kc_{2i-1}h_{2i-1}(x)
       +
       \sum_{i=1}^\infty c_{2i}h_{2i}(x).
\end{align*}
where $\gamma_1$ and $\gamma_2$ are the unique positive coefficients that turn
$\tilde{f}_1$ and $\tilde{f}_2$ into probability densities.
Then $\tilde{f}_1, \tilde{f}_2$ and $g$ are all comparable in the tail order and we
get $\tilde{f}_1 \preceq g \preceq \tilde{f}_2$.
But $g$ and $f=\frac{1}{2}\tilde{f}_1+\frac{1}{2}\tilde{f}_2$ are not comparable.
So even if  in a game all payoffs corresponding to pure strategies  are comparable, the payoffs
for mixed strategies can still be non-comparable.
\end{rem}

Now one could hope to salvage the totality of the tail order
by restricting to a space of densities which are not oscillating
"too much". The following example, however, removes this hope
as we will construct an incomparable pair of continuous densities
which each have only one local maximum, hence showing no oscillatory
behaviour on their own at all.

\begin{thm}\label{exam:2}
 For every $a,b \in \R$ with $0<a<b$ there exist a pair of absolutely continuous probability measures $\mu_f, \mu_g$ supported on $[a,b]$ with 
 unimodal $C^\infty$ densities $f,g$ and a strictly increasing sequence of positive integers
 $(\ell_n)_{n \in \N}$ such that 
 $s_{\ell_n}(\mu_f) < s_{\ell_n}(\mu_g)$ for odd $n \in \N$ and
 $s_{\ell_n}(\mu_f) > s_{\ell_n}(\mu_g)$ for even $n \in \N$.
 \end{thm}
\begin{proof}
    Let $h$ be a nonnegative $C^\infty$ unimodal function with support in the interval $(a,b)$ and integral $1$.
    Then there is a subinterval $(c,d)$ of $(a,b)$
    and a constant $K>0$
        such that $h'(x)<-K$ for all $x \in (c,d)$.
    Applying \Cref{thm:cinfty-notcomparable}, there exist $C^\infty$ functions $\tilde{f}$ and $\tilde{g}$ with support in the interval $(c,d)$ satisfying the assertion of 
    \Cref{thm:cinfty-notcomparable}.
    Since $\tilde{f}$ and $\tilde{g}$ have continuous first derivatives, we find $L>0$ such that
    $\tilde{f}'(x)<L$ and $\tilde{g}'(x)<L$ for all $x \in [c,d]$.
    Set $\alpha:=\frac{K}{L+K}$, $f:=\alpha \tilde{f}+(1-\alpha)h$
    and $g:=\alpha \tilde{g}+(1-\alpha)h$.
    Then $\alpha L + (1-\alpha)(-K) =0$, hence
    $f'(x)<0$ and  $g'(x)<0$ for $x \in (c,d)$.
    With this construction the signs of the derivatives of 
    $f$, $g$ and $h$ are identical everywhere on $\R_{\geq 0}$,
    since $\tilde f$ and $\tilde g$ have support in $(c,d)$.
    The unimodality of $h$
    is therefore inherited by $f$ and $g$.
    Furthermore $f$ and $g$ are convex combinations
    of $C^\infty$ probability densities and therefore
    $C^\infty$ probability densities themselves. 
    Finally 
    \begin{align*}
        s_{\ell_n}(\mu_f) - s_{\ell_n}(\mu_g) &= \int_{a}^b t^n (f(t)-g(t)) \,dt \\
        &= \int_{a}^b t^n \alpha (\tilde{f}(t)-\tilde{g}(t)) \,dt = \alpha (s_{\ell_n}(\mu_{\tilde{f}}) -s_{\ell_n}(\mu_{\tilde{g}}) ). 
    \end{align*}
    This shows that $f,g$ inherit the alternating behavior of the moment sequences of $\tilde{f}$ and $\tilde{g}$ obtained in \Cref{thm:cinfty-notcomparable}. The proof is complete.
\end{proof}

As a consequence we see that the tail order is not total on the following rather benign class of probability measures.

\begin{cor}
\label{cor:not-total-cinfty-unimodal}
    The tail order $\preceq$ is not total when restricted to
    probability measures with compact support in $\R_{\geq 0}$ and
    unimodal $C^\infty$ densities.
\end{cor}

 \section{Games with measure payoffs}
 \label{sec:games}
 
 The motivation for introducing the order $\preceq$ in \cite{rass2020gametheoretic_p1_v5} was to apply it as a preference ordering in non-cooperative games with probability distribution payoffs as discussed in \cite{rass2020gametheoretic_p2_v2,rassDefendingAdvancedPersistent2017}.
 For this purpose, the theory of non-cooperative games, usually formulated with real payoffs, can be generalized as follows.

     A (finite) \emph{distribution-valued  game} is a non-cooperative $n$-player game where Player $k$ has a finite set of strategies $V_k,k=1,\ldots,n$, and $V:=\prod_{k=1}^n V_k$ is the set of strategy profiles, i.e. the set of the possible ways of playing the game. Each player has a  payoff or utility function $u_k: V\to\Dco$.
     The strategies in the set $V_k$ are called pure strategies.
     A mixed strategy of Player $k$ is a formal convex combination of the pure strategies in $V_k$. The space of mixed strategies of Player $k$ is denoted $\Sigma(U_k)$. The $u_k$ are extended to the $\prod_{k=1}^n \Sigma(V_k)$ by affine extension as usual. 
     
For a strategy profile $v=(v_1, \dots, v_n)$ and a Player $k$ we will write $v=(v_k,v_{-k})$, where $v_{-k}$  contains the strategies played by the other players. In the following we restrict ourselves to the case of two-player games and we assume that both players receive identical payoffs. The players are antagonistic in the sense that while Player 1 wishes to maximize in the tail order, Player 2 aims for small payoffs. This set-up mimics the classic notion of zero-sum games. 
    A \emph{Nash equilibrium with respect to $\preceq$} is a (mixed) strategy profile $v=(v_1, v_2)$ such that 
    \begin{align*}
        \forall \tilde{v}_1 \in \Sigma(V_1)&: u_1(\tilde{v}_1, v_{-1}) \preceq u_1(v_1, v_{-1}),\\
        \forall \tilde{v}_2 \in \Sigma(V_2)&: u_2(\tilde{v}_2, v_{-2}) \succeq u_2(v_2, v_{-2}).
    \end{align*}

 In the theory of non-cooperative real-valued games, Nash's theorem states that every finite game has at least one Nash equilibrium in mixed strategies \cite{nashEquilibriumPoints1950}, \cite[Theorem 1.1]{fudenbergTiroleGameTheory1991}.
 It turns out that distribution-valued games with tail order preferences lack this important property: There are distribution-valued games that have no Nash equilibria when preferences are expressed by the tail order.
 This even holds true if we only consider payoff distributions with finite support, in which case $\preceq$ is a total order on the payoffs.

In the following, we represent a probability distribution $P$ supported on the finite set $[N] \coloneqq \{1, 2, \dots, N\}$, $N \in \N_{\geq 1}$, by the vector of probability masses $(P(1), \dots, P(N))$.
 In this representation, the tail order is equivalent to a lexicographic order which compares from right to left, see \cite[Theorem~3]{rassDecisionsUncertainConsequences2016a}.
 If $G$ is a distribution-valued game with payoffs supported on $[N]$, we define its \emph{$i$th coordinate projected game} $G_i$, $i \in [N]$, as the real-valued game where the payoff of Player $k$ under strategy profile $s$ is the $i$th coordinate of the corresponding payoff in $G$.
 
 \begin{thm}[{\cite[Example 4.40]{buergin2020bachelorThesis}}]
 \label{thm:noNash}
 There exists a distribution-valued game $G$ such that
 \begin{enumerate}[(i)]
     \item $G$ is a two-player 
     game with tail-order preferences,
     \item the payoffs of $G$ are distributions with finite support,
     \item $G$ has no Nash equilibrium.
 \end{enumerate}
 \label{thm:distribution-valued-game}
 \end{thm}
 \begin{proof}
 We will show that the distribution-valued two-player zero-sum game shown in \autoref{gameWithNoNashEquilibrium} has no Nash equilibrium with respect to the tail order. Player~1 plays the strategies $a_1, a_2$, Player~2 plays the strategies $b_1, b_2$. By definition, Player~1 prefers payoffs that are large in the tail order, while Player~2 prefers payoffs that are small in the tail order. The projected real-valued games are shown in \autoref{gameWithNoNashEquilibrium-t2}.
 \begin{table}[h]
     \centering
     \begin{tabular}{c|c|c}
                & $b_1$ & $b_2$ \\\hline
          $a_1$ & (0.3, 0.2, 0.5) & (0.6, 0.3, 0.1) \\\hline
          $a_2$ & (0.8, 0.1, 0.1) & (0.3, 0.2, 0.5)
     \end{tabular}
     \caption{Payoffs of the distribution-valued game $G$. The common support of the distributions is $\{1,2,3\}$. Depicted are the values of the probability mass functions.}
     \label{gameWithNoNashEquilibrium}
 \end{table}
 \begin{table}[h]
     \centering
     \begin{tabular}{c|c|c}
                & $b_1$ & $b_2$ \\\hline
          $a_1$ & 0.3 & 0.6 \\\hline
          $a_2$ & 0.8 & 0.3
     \end{tabular} \quad\quad
     \begin{tabular}{c|c|c}
                & $b_1$ & $b_2$ \\\hline
          $a_1$ & 0.2 & 0.3 \\\hline
          $a_2$ & 0.1 & 0.2
     \end{tabular} \quad\quad
     \begin{tabular}{c|c|c}
                & $b_1$ & $b_2$ \\\hline
          $a_1$ &  0.5 & 0.1 \\\hline
          $a_2$ & 0.1 & 0.5
     \end{tabular}
     \caption{Payoffs for the real-valued subgames $G_1$, $G_2$, $G_3$}
     \label{gameWithNoNashEquilibrium-t2}
 \end{table}
 The unique Nash equilibrium of $G_3$ is the mixed-strategy equilibrium $((0.5, 0.5), (0.5, 0.5))$.
 The projected game $G_2$ has the unique pure Nash equilibrium $((1, 0), (1, 0))$.
We will show that, together, the Nash equilibria of $G_2,G_3$
 imply that $G$ 
 does not have a Nash equilibrium.
 
 The key observation here is the following: If $v$ is a Nash equilibrium of $G$, then it is also a Nash equilibrium of $G_3$. If this were not the case, one player, without loss of generality Player $1$, could improve his payoff $u_{1,3}$ in $G_3$ by deviating from $v_k$ to some $\tilde{v}_k$. As we consider the lexicographic order from the right, 
 we have that $u_{1,3}(\tilde{v}_1,v_{-1}) > u_{1,3}(v_1,v_{-1})$ implies for the game $G$ that $u_{1}(\tilde{v}_1,v_{-1}) \succ u_{1}(v_1,v_{-1})$, a contradiction.
 
 
 As $G_3$ has the unique Nash equilibrium $v = ((0.5, 0.5), (0.5, 0.5))$, the only potential Nash equilibrium for $G$ is $v$.
 Recall that for mixed Nash equilibria in real-valued games, the following holds:
 If Player $k$ deviates to a different mixed strategy $\tilde{v}_k$ which is a convex combination of the same pure strategies used to represent $v_k$, the payoff $u_k(\tilde{v}_k, v_{-k}) = u_k(v_k, v_{-k})$ does not change \cite[cf. Theorem 2.1]{papadimitrouAlgorithmicGameTheoryChTwo2007}.
 Therefore, any deviation from $v$ by a single player keeps this player's third-coordinate payoff constant.
 But while no such deviation can improve a player's payoff in the third coordinate, it can improve it in the second coordinate:
 In particular, $u_1((1, 0), (0.5, 0.5)) = (0.35, 0.25, 0.3)$ is greater than $u_1((0.5, 0.5), (0.5, 0.5)) = (0.5, 0.2, 0.3)$ if compared lexicographically from right to left.
 So $v = ((0.5, 0.5), (0.5, 0.5))$ is not a Nash equilibrium of $G$ and so $G$ has no Nash equilibrium.
 %
 \end{proof}
 It should be noted that the previous example did not have to be constructed particularly carefully: If payoffs are chosen arbitrarily, there is a good chance to define a game without a Nash equilibrium. To justify this, we define the following hierarchy of games for a distribution-valued game $G$ with payoffs supported on $[N]$: Let $G_N$, the projected game corresponding to $N$, have a Nash equilibrium $v_N$ (which exists by Nash's theorem). Then consider $G_{N-1,v_N}$ as the projected game corresponding to $N-1$, where in addition the strategies of all players are restricted to those strategies that occur with positive weight in $v_N$ (effectively some rows and columns in the respective tables are removed). Consider a Nash equilibrium $v_{N-1}$ of  $G_{N-1,v_N}$ and construct $G_{N-2,v_N,v_{N-1}}$ in the same fashion, etc.
 For the game $G$ to have no Nash equilibrium, it is sufficient that for each Nash equilibrium $v_N$ of $G_N$, the second-highest coordinate game $G_{N-1,v_N}$, does not have $v_N$ as a Nash equilibrium, \cite[Theorem 4.47]{buergin2020bachelorThesis}.

\section{Conclusion}
\label{sec:conclusion}

We have analyzed a recently introduced partial order on the set of compactly supported finite Borel measures on the positive reals. Some cases have been identified where this order is total, but in quite benign cases the order fails to be total. Some consequences for the theory of distribution-valued games have been discussed. In general there are significant obstacles to the existence of Nash equilibria in this context. Thus in the analysis of such games, this question needs to be carefully considered. In many cases, it appears to be adequate to follow the classic route of studying Pareto or other concepts of optimality.
  
\section*{Acknowledgements:} The authors would like to thank Ali Alshawish, Hermann de Meer, Sandra König and Stefan Rass for useful discussions in the course of the preparation of this paper. In addition, we thank an anonymous reviewer for suggesting a simplified proof for the previous version of Theorem~\ref{exam:2}, which led to the current stronger statement.

 \bibliographystyle{elsarticle-num} 
 \bibliography{bibliography}
\end{document}